\documentclass[10pt]{amsart}
\usepackage[]{amsmath, amsthm, amsfonts}
\usepackage{graphicx}
\usepackage[all]{xy}
\usepackage{xcolor}
\usepackage{tikz-cd}
\usepackage{hyperref}
\usepackage{adjustbox}
\usepackage[
backend=biber,
style=alphabetic
]{biblatex}

\usepackage{cleveref}

\newcommand {\C} {{\mathbb C}}

\newcommand {\Z} {{\mathbb Z}}
\newcommand {\Q} {{\mathbb Q}}
\newcommand {\PP} {{\mathbb P}}
\newcommand {\dt} {\bullet}
\newcommand {\F} {{\mathcal F}}

\newcommand {\cH} {{\mathcal H}}

\newcommand {\A} {{\mathbb A}}

\newcommand {\OO} {{\mathcal O}}

\newcommand {\D} {\mathbb{D}}

\newcommand{\bH}{\textbf{H}}
\newcommand {\cO} {{\mathcal O}}

\newcommand{\bR}{\textbf{R}}
\newcommand {\cE} {{\mathcal E}}
\newcommand {\cF} {{\mathcal F}}
\newcommand{\cD}{\mathcal{D}}
\newcommand {\cL} {{\mathcal L}}
\newcommand {\cM} {{\mathcal M}}
\newcommand{\cN}{\mathcal{N}}
\newcommand{\bT}{\begin{tikzcd}}
\newcommand{\eT}{\end{tikzcd}}
\newcommand{\GR}{Gr^{F}_{p}DR(\cN)}
\newcommand{\tX}{\tilde{X}}
\newcommand{\tY}{\tilde{Y}}
\newcommand{\tM}{\tilde{M}}
\newcommand{\tcM}{\tilde{\cM}}
\newcommand{\tcN}{\tilde{\cN}}
\newcommand{\tK}{\tilde{K}}

\DeclareMathOperator{\im}{im}

\DeclareMathOperator{\Spec}{Spec}

\DeclareMathOperator{\Proj}{Proj}
\DeclareMathOperator{\depth}{depth}

\newtheorem{thm}[subsection]{Theorem}
\newtheorem{cor}[subsection]{Corollary}
\newtheorem{lemma}[subsection]{Lemma}
\newtheorem{prop}[subsection]{Proposition}
\newtheorem{defn}[subsection]{Definition}
\newtheorem{rmk}[subsection]{Remark}
\newtheorem{ex}[subsection]{Example}

\newtheorem{nota}[subsection]{Notation}
\newtheorem{set}[subsection]{Setting}

\begin{document}

\author{Donu Arapura and
Scott Hiatt }
\date{\today}
  \address{
 Department of Mathematics\\
  Purdue University\\
  West Lafayette, IN 47907\\
  U.S.A.}
  \address{
 Department of Mathematics\\
  DePauw University\\
  Greencastle, IN 46135\\
  U.S.A.}
  \thanks{First author supported by a grant from the Simon's foundation}

\title{Equivariant Hodge modules and rational singularities}


\maketitle

\begin{abstract}
 We define a notion of Hodge modules with rational singularities. A variety has rational singularities in the usual sense, if
 it is normal and the Hodge module related to intersection cohomology has rational singularities in the present sense.
 Our main result is a generalization of Boutot's theorem that if a reductive group acts on a smooth affine variety with a stable point,
 and $\cM$ is an equivariant Hodge module with rational singularities, then the induced module on the GIT quotient also
 has rational singularities.
\end{abstract}

The goal of this paper is to develop a theory of  rational singularities with coefficients,
and to give an extension of Boutot's theorem  \cite{boutot} to this setting.
Before we  state the results precisely, we recall that a Hodge module  $\cM$ on a smooth variety $X$
consist of a $\cD$-module $M$ with a good filtration $F_\dt$  and a compatible perverse sheaf of $\Q$-vector spaces satisfying appropriate
conditions \cite{saito1}.
A simple example is $\Q_X^H [ \dim X]$, which consists of the right $\cD$-module $\Omega^{\dim X}_{X}$ with trivial filtration, and
perverse sheaf $\Q_{X}[\dim X]$. Hodge modules are also defined when $X$ is singular. If $X_{\text{sm}}\subseteq X$  is the smooth locus,
the module $\Q_{X_{\text{sm}}}^H [\dim X]$ has a unique
extension $IC_X^H$  to a Hodge module on  $X$, such that the underlying perverse sheaf is the intersection cohomology complex.
Associated to a Hodge module $\cM$ is the 
$\OO_X$-module $S_{X}(\cM)=F_p(M) \neq 0$, where $p$ is chosen minimal. We also have a complex of $\OO_X$-modules $Q_{X}(\cM)$, which is roughly dual to $S_{X}(\cM)$; more precisely,  $Q_{X}(\cM)= \bR \cH om_{X}( S_{X}( \cM), \omega^{\bullet}_{X})$. Suppose that 
$\cM$ is a Hodge module strictly supported on $X$, which means  that  it supported on $X$, and 
it  has no nontrivial  subquotient  supported on a proper subvariety of $X$.
We define $\cM$ to have  {\em rational singularities} if $Q_{X}(\cM)$ is quasi-isomorphic to a sheaf placed in degree $-\dim X$, or
equivalently if $S_{X}(\cM)$ is a maximal Cohen-Macaulay module. 
We can see that $\Q_X^H [\dim X]$ has rational singularities when $X$ is smooth, since $Q_{X}(\Q_X^H[\dim X])= \OO_X[\dim X]$.
More generally, $X$ has rational singularities in the usual sense  if and only if $X$ is normal and
$IC_X^H$ has rational singularities in the present sense (see Corollary \ref{cor:rs}).  This gives a characterization of rational singularities which does not depend on
a resolution of singularities. 

Given a regular $\C$-algebra $R$ of finite type with an action of a reductive group $G$, Hochster and Roberts \cite{hr} proved that
the ring of invariants $R^G$ is Cohen-Macaulay.  The theorem was refined by Boutot \cite{boutot} who proved that $\Spec R^G$
has rational singularities. It seems natural to consider analogues of these results for modules.
Suppose that $G$ is a reductive group that acts effectively on a smooth affine variety $X=\Spec R$ such that $X$ has a stable point,
then we show that there is an equivalence of categories between the category of equivariant Hodge modules   strictly supported on $X$,
and Hodge modules strictly supported on $Y= \Spec R^G$. 
\begin{prop} Suppose that $G$ is a reductive group that acts effectively on $X$ such that there exists a stable point. Let $Y = X//G$.
 Then there is an equivalence of categories between $HM_{Y}(Y,w)$ and $ HM_{X,G}(X, w+d)$ .
\end{prop}
Suppose that $\cM$ is an equivariant Hodge module  strictly support on $X$,
and let $\pi^{G}_{+}\cM$ be the corresponding Hodge module on $Y$ under this equivalence.
Our generalization of Boutot's theorem is as follows:

\begin{thm}
     Let $G$ be a $d$-dimensional reductive algebraic group  with an effective action on a smooth affine variety $X$ such that there exists a stable point.  Let $\pi:X\to Y= X//G$ be the quotient map. 
     Suppose that  $\cM \in HM_{X,G}(X,w+d)$ has rational singularities, then $\cN = \pi^{G}_{+}\cM \in HM_{Y}(Y,w)$ has rational singularities, and  
     $$Q_{Y}(\cN)[d] \cong (\pi_{*}Q_{X}(\cM))^{G}$$
     $$S_{Y}(\cN) \cong (\pi_{*}S_{X}(\cM))^{G}.$$
\end{thm}

We recover (a special case of) Boutot's theorem by applying our theorem to $\cM= \Q_X^H[\dim X]$.
It implies that  the corresponding module $\pi^{G}_{+}\cM = IC_Y^H$ has
rational singularities. Since $Y$ is easily seen to be normal, it must therefore have rational singularities by the result stated in the previous paragraph.

The analogue of Hochster-Roberts  for modules was considered   by Stanley \cite{stanley} and Van den Bergh \cite{bergh}.
When $A$ is a finitely generated free $R$-module on which $G$ acts compatibly,  they give some criteria when 
$A^G$ is Cohen-Macaulay, and they  give simple examples where this   can fail. As a corollary of our main theorem, we deduce a different criterion:
 
 \begin{cor}
Let $X=\Spec R$ be a smooth variety admitting  an action by a reductive group $G$ such that the stable locus $X^s\not=\emptyset$.
 Let $V$ be  a rational representation of $G$. Then $(V\otimes_\C R)^G$  is Cohen-Macaulay  provided  that $V\otimes R$ is a direct summand of $Q_{X}(\cM)[-\dim X]$ for
 some $\cM\in HM_{X,G}(X)$ with RS.
\end{cor}

We give an explicit special case when $R$ is a polynomial ring in Example \ref{ex:CM}.

We work with complex algebraic varieties with their analytic topologies throughout this paper. The material from the first section of this paper
is drawn from the second author's thesis \cite{hiattT}. Our thanks to the referee for careful reading, and for many helpful comments.


\section{Rational Singularities for Pure Hodge Modules}
Let $X$ be an irreducible complex variety. Assume $U \subset X$ is a smooth open subset, and let $\bH$ be a  variation of Hodge structure 
(always assumed polarizable) on $U.$  Koll\'ar  \cite[\S 5]{kollar1} conjectured that there were objects $S(X, \bH)$ and $Q(X,\bH)$ that should behave similar to $\omega_{X}$ and $\cO_{X}$ when $X$ is smooth. Specifically, $S(X, \bH)$ and $Q(X,\bH)$ should have the following properties:
\begin{itemize}

\item $S(X, \bH)$ is a coherent sheaf and $Q(X, \bH)$ is in the bounded coherent derived category $D^{b}_{coh}(X)$

\item Assume $f: X' \rightarrow X$ is a birational projective map. Since $f$ is birational, $X$ and $X'$ contain isomorphic open subsets. If $\bH'$ is the pullback of the variation of Hodge structure $\bH$, then 
$$\bR f_{*}S(X',\bH') = S(X, \bH) \quad \text{and} \quad  \bR f_{*}Q(X', \bH') = Q(X,\bH).$$

\item If $X$ is smooth and $\bH$ gives a variation of Hodge structure over an open subset $U \subset X$ such that $X \backslash U$ is a divisor with normal crossings only, then 
$$S(X, \bH) = \omega_{X} \otimes \hfill ^{u}\cF^{b}(\bH) \quad  \text{and} \quad  Q(X, \bH) =  \hfill ^{l}Gr^{0}(\bH).$$
See \cite[Def. 2.3]{kollar1} for the definition of $^{u}\cF^{b}(\bH)$ and $^{l}Gr^{0}(\bH)$.

\item If $\bH^{*}$ is the dual of $\bH$, then $S(X, \bH^{*}) =  \cD(Q(X, \bH))$ and $Q(X, \bH^{*}) = \cD(S(X,\bH))$, where $\cD = \bR \cH om( \bullet, \omega^{\bullet}_{X}).$
\end{itemize}

\begin{nota}
In this paper, $\omega^{\bullet}_{X} \in D^{b}_{coh}(\cO_{X}))$ will denote the dualizing complex of $X$. Note that when $X$ is smooth, $\omega^{\bullet}_{X} \simeq \Omega^{\dim X}_{X}[\dim X]$. 
\end{nota}

This conjecture was proven by M. Saito \cite{saito3} using properties of Hodge modules. This section will be devoted to reviewing Saito's results \cite{saito3}, and applying these results to define rational singularities for Hodge modules.

Although we will mostly treat Hodge module theory as a black box,
we will need to recall a few details. 
 We refer to Saito's foundational papers \cite{saito1, saito2}, and to Schnell's overview \cite{schnell} for a more precise account.
  Let $X$ denote an algebraic variety of dimension $n$. 
  For simplicity, in this section, we will always assume $X$ embeds into a smooth variety $Y$ of dimension $m$,  and let 
  $$\bT 
  i: X \arrow[r, hook] & Y
  \eT$$
denote the embedding. Saito associates to $X$ a semisimple $\Q$-linear abelian category $HM(X)$
of  Hodge modules, such that every simple object is generically a  variation of Hodge structure on a closed irreducible subvariety of $X$.
For the semisimplicity statement, we need to  assume that Hodge modules are polarizable.
To be a bit more explicit, a Hodge module consists of a right  regular holonomic $\cD_Y$-module $M$ supported
on $X$, a good filtration $F_\dt$ on $M$, and a $\Q$-structure provided by a perverse sheaf $L$ of $\Q$-vector spaces and an isomorphism $DR(M)\cong \C\otimes L$
subject to additional conditions.
Although $M$ is a $\cD_Y$-module, the associated graded of the  de Rham complex with respect to the induced filtration
$$F_pDR(M) = [F_{p- m}M\otimes \wedge^{m} T_Y \to F_{p-m+1}M\otimes \wedge^{m-1} T_Y\to \ldots F_pM][m]$$
 is a complex of $\OO_X$-modules. We refer the reader to \cite{htt} for more details into $\cD$-modules. Below, we find it convenient to refer to $\cD_Y$-modules supported on $X$
 simply as $\cD_X$-modules. When $X$ is smooth, these are indeed the same thing by a theorem of Kashiwara \cite[Theorem 1.6.1]{htt}.
 
 The category $HM(X)$ has a decomposition
 $$HM(X) = \bigoplus_{w\in \Z, Z\subseteq X}HM_Z(X,w)$$
 where $HM_Z(X,w)$ is the category of Hodge modules of
 weight $w$ and strict support  along a closed irreducible subvariety $Z\subset X$. A Hodge module is strictly supported on a closed subvariety  $Z$, when  it is supported on  $Z$ and it  has no nontrivial subquotient  supported on a proper subvariety of $Z$.
 One of Saito's key results \cite[theorem 3.21]{saito2} is
that $HM_Z(X, w)$ is equivalent to the category of  variations of Hodge structure of weight $w-\dim Z$ generically defined on $Z$.
If $\bH $ is a variation of Hodge structure defined on a Zariski open $j:U\to Z$, it extends to a Hodge module with underlying
perverse sheaf given by intersection cohomology complex $j_{!*}L[\dim Z]$.
Saito \cite{saito2} constructs a bigger abelian category of mixed Hodge modules
 $MHM(X)\supset HM(X)$. A mixed Hodge module $\cN \in MHM(X)$ contains a finite, increasing weight filtration $W_{\bullet}\cN$, where $Gr^{W}_{i}\cN$ is a pure Hodge module of weight $i$ on $X$. For a mixed Hodge module $\cN \in MHM(X)$, we will denote the above data as 
$$\cN = ((N, F_{\bullet}N), W_{\bullet}\cN, L),$$
where $N$ is the underlying $\cD_{X}$-module and $L$ is the perverse sheaf. The weight filtration will be omitted if $\cN$ is pure of weight $w$. It is worth noting that the derived
category of mixed Hodge modules has standard operations such as direct images, denoted here by $f_+$, and these are compatible with the corresponding operations  on the
constructible derived category.
 
We begin with a slightly generalized version of $S(X, \bH)$ and $Q(X, \bH)$ given above.

\begin{defn}\label{SQdefn} Let $\cN \in MHM(X)$ be a mixed Hodge module. If $p(\cN) = \text{min}\{p: F_{p}N \neq 0\}$, where $N$ is the underlying $\cD_{X}$-module of $\cN$, then  set
\begin{equation}
S_{X}(\cN) := Gr^{F}_{p(\cN)} DR(\cN) = F_{p(\cN)}N
\end{equation}
\begin{equation}
 Q_{X}(\cN):= \cD(S_{X}(\D(\cN))) \in D^{b}_{coh}(X),
 \end{equation}
 
where $\cD(\bullet) = \bR \cH om_{X}( \bullet, \omega^{\bullet}_{X})$ and $\D(\cN)$ is the dual Hodge Module. 
\end{defn}

\begin{prop}\label{dualprop}\cite{saito1}
If $\cN \in MHM(X)$, then we have the quasi-isomorphism
$$ \bR \cH om_{X}(Gr^{F}_{p}DR(\cN), \omega^{\bullet}_{X}) \simeq Gr^{F}_{-p}DR(\D(\cN)) \quad  \text{for every $p \in \Z$}, $$ 
where $\D(\cN)$ is the dual Hodge module.
\end{prop}

\begin{rmk}\label{rmk:polar}
 When $ \cN \in HM(X,w)$ is a pure Hodge module of weight $w$,
 then a polarization gives an isomorphism $\D(\cN) \cong \cN(w),$
 where 
$$\cN(w)= ((N, F_{\bullet -w}N), W_{\bullet-2w}\cN, L \otimes_{\Q} \Q(w))$$
is the Tate twist of $\cN$ by the weight $w$. Therefore, $Q_{X}(\cN)= \cD(S_{X}(\cN)). $
  
\end{rmk}

\begin{lemma}\label{Q-lemma} If $\cN \in HM(X,w)$ is a pure Hodge module of weight $w$, then $Q_{X}(\cN) \simeq Gr^{F}_{-p(\cN) -w}DR(\cN)$. 
\end{lemma}

\begin{proof}
By the previous remark, $\D(\cN) \cong \cN(w)$. Therefore, $p(\D(\cN)) =  p(\cN)+w$. Applying Proposition \ref{dualprop} we obtain
$$Q_{X}(\cN) = \cD(Gr^{F}_{p(\cN) }DR(\cN)) = \cD(Gr^{F}_{p(\cN) +w}DR(\cN(w))) \simeq Gr^{F}_{-p(\cN) -w}DR(\D(\cN(w)))$$
$$=  Gr^{F}_{-p(\cN) -w} DR(\cN) .$$ 
\end{proof}

\begin{rmk}
If $X$ is irreducible, and  $\cN \in HM_{X}(X,w)$ is a Hodge module strictly supported on $X$,  then there exists a variation of Hodge structure $\bH$ on a smooth open subset $U \subset X$ such that $\cN$ restricted to $U$ is isomorphic to $\bH$ \cite[Thm. 3.21]{saito2}. In this situation, $S_{X}(\cN)= S(X,\bH)$ and $Q_{X}(\cN) = Q(X, \bH)$ satisfy the conditions conjectured above.
\end{rmk}

\begin{ex}
Assume $X$ is a smooth irreducible variety. If  $\cN =\Q_{X}^{H}[n],$ then $\cN$ is a  pure Hodge module of weight $n,$ and 
$$p(\Q^{H}_{X}[n]) = -n \quad \text{and} \quad S_{X}(\Q^{H}_{X}[n]) = \Omega^{n}_{X}.$$
By Lemma \ref{Q-lemma}, we have 
$$Q_{X}(\Q^{H}_{X}[n]) = Gr^{F}_{0}(DR(\Q^{H}_{X}[n])) = \mathcal{O}_{X}[n].$$
\end{ex}

\begin{prop}\label{MinProp}\cite[Prop. 4.7]{ks} If $X$ is smooth, $\cN \in MHM(X)$, and $\alpha = min\{w: W_{w}\cN \neq 0\}$, then 
$$Gr^{F}_{p}DR(\cN) \simeq 0 \quad \text{unless $p(\cN) \leq p \leq  - p(\cN) -\alpha $}. $$
\end{prop}

\begin{proof}
By definition of $p(\cN)$, we know $\GR = 0$ if $p < p(\cN)$. So, we just have to show the other inequality.

When $\cN$ is a pure Hodge module of weight $\alpha$, there is a quasi-isomorphism
$$\bR \cH om_{X}(\GR,\omega^{\bullet}_ {X}) \simeq Gr^{F}_{-p -\alpha}DR(\cN).$$
If $p > -p(\cN) -\alpha$, then $-p  - \alpha < p(\cN)$ and  $ Gr^{F}_{-p -\alpha}DR(\cN) = 0$. The quasi-isomorphism above implies $\GR \simeq 0$ for $p > -p(\cN) -\alpha$.

Now, we have the exact sequence of mixed Hodge modules
$$0 \rightarrow Gr^{W}_{\alpha}\cN \rightarrow \cN \rightarrow \cN/W_{\alpha}\cN \rightarrow 0.$$
There is an exact triangle
$$\bT 
Gr^{F}_{p}DR(Gr^{W}_{\alpha}\cN) \arrow[r] & Gr^{F}_{p}DR(\cN) \arrow[r] & Gr^{F}_{p}DR(\cN / W_{\alpha}\cN) \arrow[r, "+1"] & \hfill.
\eT$$
If $p > - p(\cN) - \alpha$, then $Gr^{F}_{p}DR(Gr^{W}_{\alpha}\cN) \simeq 0$. So it suffices to show $ Gr^{F}_{p}DR(\cN / W_{\alpha}\cN) = 0$ for $p > -p(\cN) - \alpha$.
There is an exact triangle
$$\bT 
Gr^{F}_{p}DR(Gr^{W}_{\alpha +1}\cN) \arrow[r] & Gr^{F}_{p}DR(\cN / W_{\alpha}\cN) \arrow[r] & Gr^{F}_{p}DR(\cN / W_{\alpha+1}\cN) \arrow[r, "+1"] & \hfill.
\eT$$
 For the mixed Hodge module $\cN/W_{\alpha}\cN$, we must have $p(\cN/W_{\alpha}\cN) \geq p(\cN)$. So we must have $p(Gr^{W}_{\alpha +1}\cN) \geq p(\cN)$ as well. This forces $Gr^{F}_{p}DR(Gr^{W}_{\alpha +1}\cN) \simeq 0$ for $p> -p(Gr^{W}_{\alpha +1}\cN) - \alpha -1.$ In particular $Gr^{F}_{p}DR(Gr^{W}_{\alpha +1}\cN) \simeq 0$ for $p > -p(\cN) - \alpha.$ So now it suffices to show $ Gr^{F}_{p}DR(\cN / W_{\alpha +1}\cN) = 0$ for $p > -p(\cN) - \alpha$. Continuing in this same fashion, and using the fact that $Gr^{F}_{p}DR(Gr^{W}_{\alpha +i}\cN) \simeq 0$ whenever $i >0$ and $p> -p(\cN) - \alpha$, we must have $Gr^{F}_{p}DR(\cN) \simeq 0$ for $p > - p(\cN) - \alpha$ because the weight filtration is finite.

\end{proof}

 Assume $X$ is a smooth irreducible variety. Let $\cN \in HM_{X}(X,w)$ and assume the singularities of  $\cN$  are  along a simple normal crossing divisor $D$, i.e.   suppose that  the restriction of $\cN$  to $X \backslash D$ is  a  variation of Hodge structure $\bH= ((\cL, F^{\bullet}), L)$ of weight $w - n$ on the open set $U = X \backslash D.$ Let $\iota: D \hookrightarrow X$ denote the natural map. If $j:U \hookrightarrow X$ is the natural map, recall that the Hodge filtration on $\cN$ is given by
\begin{equation}
\displaystyle F_{p}\cN = \sum_{i \geq 0} \bigg( \Omega^{n}_{X} \otimes ( \tilde{\cL}^{>-1} \cap j_{\ast}F^{i -p -n}\cL) \bigg) F_{i} \cD_{X},
\end{equation}
where  $\tilde{\cL}^{> \alpha}$  is Deligne's extension of $\cL$ with eigenvalues of residue of connection in $(\alpha, \alpha +1]$ \cite{saito4}.

\begin{prop}\label{prop:ncd}
For the Hodge module $\cN$ above, we have
\begin{equation}
S_{X}(\cN) = \Omega^{n}_{X} \otimes (\tilde{\cL}^{>-1} \cap j_{\ast}F^{-p(\cN) -n }\cL).
\end{equation}

\begin{equation}
Q_{X}(\cN) \simeq \tilde{\cL}^{\geq 0} / (\tilde{\cL}^{\geq 0} \cap j_{\ast}F^{w + p(\cN) +1}\cL)[n].
\end{equation}
\end{prop}

\begin{proof}

This is discussed in \cite{saito3}. For convenience to the reader, a proof will be given.  

Equation (4) is clear by the definition of the Hodge filtration for $\cN$. To show equation (5),  consider the exact triangle 
$$\begin{tikzcd}
\iota_{+}\iota^{!}\cN \arrow[r] & \cN \arrow[r] & j_{+}j^{*}\cN \arrow[r,"+"] & \hfill.
\end{tikzcd}$$
Since $\cN$ has no non-trivial sub-object with support on $D$, there is an exact sequence
$$ 0 \rightarrow \cN \rightarrow \cN(*D) \rightarrow \iota_{+}\cH^{1}\iota^{!}\cN \rightarrow 0.$$
For notation, let $\cH^{1}_{D}(\cN) =  \iota_{+}\cH^{1}\iota^{!}(\cN).$ Since $\cN$ has pure weight $w$, we have $Gr^{W}_{i}(\cH^{1}_{D}(\cN)) = 0$ for $i < w +1$ \cite[Prop. 2.26]{saito2}.  By Proposition \ref{MinProp}, we obtain 
$$Gr^{F}_{-p(\cN) -w}DR(\cH^{1}_{D}(\cN)) \simeq 0.$$ So, from the exact sequence above, 
$$Q_{X}(\cN) = Gr^{F}_{-p(\cN) -w}DR(\cN) \simeq Gr^{F}_{-p(\cN) -w}DR(\cN(\ast D)).$$
 
By \cite[Prop. 3.11]{saito2}, there is a quasi-isomorphism
$$Gr^{F}_{-p(\cN) -w}DR(\cN(\ast D)) \simeq Gr^{F}_{-p(\cN) -w}\tilde{\cL}^{\geq 0} \rightarrow \Omega^{1}(\log D) \otimes Gr^{F}_{-p(\cN) -w +1}\tilde{\cL}^{\geq 0} \rightarrow \cdots.$$
On the variation of Hodge structure $\bH,$ we have $p(\cN) +w = \text{min}\{p: Gr^{p}_{F}\cL \neq 0\}$. So, when $i > 0$, we must have
$$Gr^{F}_{-p(\cN) -w +i}\tilde{\cL}^{\geq 0} = (\tilde{\cL}^{\geq 0} \cap j_{\ast}F^{p(\cN) +w-i}\cL)/ (\tilde{\cL}^{\geq 0} \cap j_{\ast}F^{p(\cN) +w -i+1}\cL) = 0.$$
Hence, 
$$Q_{X}(\cN)  \simeq Gr^{F}_{-p(\cN) -w} \tilde{\cL}^{\geq 0}[n] = \tilde{\cL}^{\geq 0} / (\tilde{\cL}^{\geq 0} \cap j_{\ast}F^{w + p(\cN) +1}\cL)[n].$$
\end{proof}


Assume $X$ is an irreducible variety and $\cN \in HM_{X}(X,w)$. If $X$ is smooth, and $\cN$ is the trivial Hodge module, then we see that $Q_{X}(\cN)$ is given by a coherent sheaf (up to a shift). Similarly, if $X$ is smooth, and $\cN$ is a variation of Hodge structure on $X$, then $Q_{X}(\cN)$ is a coherent sheaf.  This continues to hold when the singularity of $\cN$ is a simple normal crossing divisor. In general, the object $Q_{X}(\cN)$ is not a coherent sheaf, but $Q_{X}(\cN) \in D^{b}_{coh}(X)$. We will investigate Hodge modules on $X$ that have the property that $Q_{X}(\cN)$ is just a coherent sheaf. 

\begin{defn}\label{defn:rs}
 If $X$ is an irreducible variety, a Hodge module $\cN \in HM_{X}(X,w)$ has rational singularities (or has RS) if $$Q_{X}(\cN) \simeq \cH^{-n}(Q_{X}(\cN))[n].$$
\end{defn}

\begin{lemma}\label{lemma:ncdrs}
  If $X$ is a smooth irreducible variety, then $\cN \in HM_{X}(X,w)$ has RS if it has singularities along a divisor with normal crossings.
\end{lemma}

\begin{proof}
 This follows from Proposition \ref{prop:ncd}.
\end{proof}

\begin{lemma}\label{lemma:SCM} $\cN \in HM_{X}(X,w)$ has RS if and only if $S_{X}(\cN)$ is Cohen-Macaulay of pure dimension $n$.
\end{lemma}

\begin{proof}

  Given the embedding $i : X \hookrightarrow Y$ with $Y$ smooth, by Remark \ref{rmk:polar} there is a quasi-isomorphism
$$ i_{\ast} Q_{X}(\cN) \simeq \bR \cH om_{Y}(i_{\ast}S_{X}(\cN), \omega^{\bullet}_{Y}).$$
We see that $Q_{X}(\cN) \simeq \cH^{-n}(Q_{X}(\cN))[n]$ if and only if 
$$\cE xt^{i}_{Y}(i_{\ast}S_{X}(\cN), \omega^{\bullet}_{Y}) = 0 \quad \text{unless $i =  -n$}.$$
Which is equivalent to  $S_{X}(\cN)$ being Cohen-Macaulay of pure dimension $n$ \cite[Chap. I, \S 3]{toroidal}. 
\end{proof}

\begin{cor}\label{locallyfree}
 Assume $X$ is smooth. If $\cN \in HM_{X}(X,w)$ has RS, then $S_{X}(\cN)$ and $Q_{X}(\cN)$ are (translates of) locally free sheaves.
\end{cor}

\begin{proof}
 Locally, $S_{X}(\cN)$ is a maximal Cohen-Macaulay module. It is well known that a maximal Cohen-Macaulay module on a regular local ring is free. 
\end{proof}

Next, we will explain a  characterization of rational singularities for Hodge modules which is closer to the original definition.
 Specifically, we will show we can define rational singularities for Hodge modules by a resolution of singularities. To do this, we will need Saito's main result from \cite{saito3}. 
Following Saito \cite{saito3}, we define
$$ q(\cN) := -n - p(\cN) \quad \text{and} \quad q'(\cN) := p(\cN) + w.$$

\begin{thm}\cite[Thm. 3.2]{saito3}\label{thm:saito3}
Let $f: X \rightarrow X'$ be a proper surjective morphism of irreducible varieties with $d =dim(X) - dim(X')$, and $\cM \in HM_{X}(X, w)$. If $\cM^{i}_{X'} \in HM_{X'}(X', w +i)$ is the direct factor of $H^{i}f_{+}\cM$ from the decomposition of strict support,  then we have the (noncanonical) isomorphisms in $D^{b}_{coh}(X')$:
$$ \bR f_{\ast} S_{X}(\cM) \simeq \displaystyle \bigoplus_{q(\cM^{i}_{X'}) = q(\cM) + d} S_{X'}(\cM_{X'}^{i})[-i],$$

$$\bR f_{\ast} Q_{X}(\cM) \simeq \displaystyle \bigoplus_{q'(\cM^{i}_{X'}) = q'(\cM) } Q_{X'}(\cM_{X'}^{i})[-i],$$

and canonical isomorphisms 
\begin{equation*}
 R^{i}f_{\ast}S_{X}(\cM) = S_{X'}(\cM_{X'}^{i}) \quad \text{for $q(\cM^{i}_{X'}) = q(\cM) + d$ }
 \end{equation*}
 
 \begin{equation*}
 ^{d}H^{i} \bR f_{\ast}Q_{X}(\cM) = Q_{X'}(\cM_{X'}^{i})\quad  \text{for $q'(\cM^{i}_{X'}) = q'(\cM),$ }
 \end{equation*} 
 
 where $^{d}H^{i}\bR f_{\ast}Q_{X}(\cM) = \cD(H^{-i}(\cD( \bR f_{\ast}Q_{X}(\cM)))).$

\end{thm}

\begin{set}\label{settingRes} Let $X$ be an irreducible variety. If  $\bH = ((\cL, F), L)$ is a variation of Hodge structure of weight $w -n$ on a smooth open Zariski dense  subset $U$ of $X$, then there exists a unique Hodge module $\cM = ((M,F), K) \in HM_{X}(X,w)$ such that $\cM|_{U} \cong \bH$.  If $\pi: \tX \rightarrow X$ is a resolution of singularities, then there exists a smooth open Zariski dense subset $V$ such that $\pi|_{\pi^{-1}(V)}:\pi^{-1}(V) \rightarrow V$ is an isomorphism. If we restrict the variation of Hodge structure $\bH$ to the open subset $U \cap V$, then there is a variation of Hodge structure $\tilde{\bH}$ on the open set $\pi^{-1}(V \cap U)$. Hence there also exists a unique Hodge module $\tcM = ((\tM, F), \tK)\in HM_{\tX}(\tX, w)$  such that $\tcM|_{\pi^{-1}(V \cap U)} \cong \tilde{\bH}$. 
\end{set}

\begin{lemma}\label{lemma:dirfact}
In the setting \ref{settingRes}, $\cM$ is the only direct factor of the Hodge module $H^{0}\pi_{+}\tcM$ with strict support on $X$.
\end{lemma}

\begin{proof}
By the condition of strict support, we have $H^{0}\pi_{+}\tcM = \bigoplus \cM_{Z}$, where $\cM_{Z} \in HM_{Z}(X, w)$. There is also an isomorphism $H^{0}\pi_{+}\tcM|_{V \cap U} \cong \bH|_{V \cap U}$. Therefore, if $\cM' \in HM_{X}(X,w)$ is the unique Hodge module that extends $\bH|_{V \cap U}$, then $\cM^{'}$ is the only direct factor of $H^{0}\pi_{+}\tcM$ with strict support on $X$. So, we need to show $\cM = \cM^{'}$. By uniqueness of $\cM$, it suffices to show $\cM'|_{U} \cong \bH$. By the Riemann-Hilbert correspondence, it suffices to show the underlying perverse sheaf of $\cM'|_{U}$  is $L$. If $j_{V \cap U}: V \cap U \hookrightarrow X$ is the natural inclusion, then the perverse sheaf of $\cM'$ is given by the intermediate extension $j_{(V\cap U)!\ast}L|_{V \cap U}$. If $j_{U}:U \hookrightarrow X$ and $j:V \cap U \hookrightarrow U$ are the natural maps, then by the properties of the intermediary extension, there are isomorphisms
$$ L \cong j_{!\ast}L|_{V \cap U} \cong j_{U}^{-1}(j_{(V\cap U)!\ast}L|_{V \cap U}).$$
See \cite[Remark 5.2.7]{dimca} for more details.

\end{proof}

\begin{prop}\label{prop:dirS}
In the setting \ref{settingRes}, there are quasi-isomorphisms
$$\bR \pi_{\ast}S_{\tX}(\tcM) \simeq
 \pi_{\ast}S_{\tX}(\tcM) \simeq S_{X}(\cM) \quad \text{and} \quad \bR \pi_{\ast}Q_{\tX}(\tcM) \simeq Q_{X}(\cM).$$
\end{prop}

\begin{proof}
From Theorem \ref{thm:saito3},
$$ \bR \pi_{\ast} S_{\tX}(\tcM) \simeq \displaystyle \bigoplus_{q(H^{i}\pi_{+}(\tcM)) = q(\tcM) } S_{X}(H^{i}\pi_{+}(\tcM))[-i].$$
But, for $i \neq 0$, the support of $H^{i}\pi_{+}(\tcM)$ is contained in a proper closed subset of $X$. Therefore, by combining Theorem \ref{thm:saito3} and Lemma \ref{lemma:dirfact}, 
$$ \bR \pi_{\ast} S_{\tX}(\tcM) \simeq \pi_{\ast}S_{\tX}(\tcM) \cong S_{X}(H^{0}\pi_{+}(\tcM)) \cong S_{X}(\cM).$$
The second quasi-isomorphism follows from duality (Remark \ref{rmk:polar}).
\end{proof}

\begin{cor}\label{cor:biratRS}
Given a Hodge module $\cM \in HM_{X}(X,w)$, choose a smooth open set $U\subset X$, such that
$\cM|_U$ is a variation of Hodge structure. Fix a resolution of singularities $\pi:\tilde X\to X$, such that
$\pi^{-1}(X-U)$ is a divisor with normal crossings. Then,
\begin{enumerate}
\item $\cM$
has RS if and only if  $R^{i} \pi_{\ast} Q_{\tX}(\tcM) = 0$  for $i > -n$; 
\item $S_{X}(\cM)$ is torsion free. 
 
\end{enumerate}
\end{cor}

\begin{proof}
The first item follows from the definition of RS for Hodge modules and Proposition \ref{prop:dirS}. For the second item, $S_{\tX}(\cM)$ is torsion free because it is locally free (Corollary \ref{locallyfree}). Hence $\pi_{*}S_{\tX}(\tcM) \cong S_{X}(\cM)$ is torsion free.
\end{proof}

\begin{cor}\label{cor:rs}
Let $U$ be the nonsingular set of $X$ and $IC^{H}_{X} \in HM_{X}(X,n)$ the unique Hodge module that extends the constant variation of Hodge structure on $U$. The variety $X$ has rational singularities if and only if $X$ is normal and $IC^{H}_{X}$ has rational singularities.
\end{cor}

\begin{proof} 
Let $\pi: \tX \rightarrow X$ be a resolution of singularities. The Hodge module $IC^{H}_{X}$ is a direct factor of $H^{0}\pi_{+}\Q^{H}_{\tX}[n]$, and
$$\bR\pi_{\ast}\cO_{\tX}[n] = \bR \pi_{\ast}Q_{\tX}(\Q^{H}_{\tX}[n]) \simeq Q_{X}(IC^{H}_{X}).$$
Hence $IC^{H}_{X}$ has RS if and only if the higher direct images of $\cO_{\tX}[n]$ vanish. So, $X$ has rational singularities if and only if $X$ is normal and $IC^{H}_{X}$ has RS.
\end{proof}

Using the definition of rational singularities for Hodge modules with Theorem \ref{thm:saito3}, we also have a theorem analogous to Koll\'ar's result \cite[Prop. 3.12]{kollar2}. 

\begin{thm}
Let $f: X \rightarrow X'$ be a surjective projective map between irreducible algebraic varieties of dimension $n$ and $m$, respectively, and let $d = n -m$. If  $\cN \in HM_{X}(X,w)$, then the following are equivalent:
\begin{enumerate}
\item $R^{i}f_{\ast}Q_{X}(\cN)$ is torsion free for all $i$. \\

\item $\bR f_{\ast} Q_{X}(\cN) \simeq \displaystyle \bigoplus_{i} R^{i}f_{\ast}Q_{X}(\cN)[-i]$. \\

\item If $\cN^{j}_{X'} \in HM_{X'}(X', w +j)$ is the largest direct factor of $H^{j}f_{+}\cN$ with strict support $X'$, then $\cN^{j}_{X'}$ has RS. 
\end{enumerate}

\end{thm}

\begin{proof}
$(1) \Rightarrow (3):$ By duality, and Theorem \ref{thm:saito3}, we have the quasi-isomorphisms
$$ \bR f_{\ast} Q_{X}(\cN) \simeq \bR \cH om_{X'}(\bR f_{\ast}S_{X}(\cN), \omega^{\bullet}_{X'}) \simeq \displaystyle \bigoplus_{0 \leq j \leq d} \bR \cH om_{X'}(R^{j}f_{\ast}S_{X}(\cN)[-j], \omega^{\bullet}_{X'}) $$ 
$$\simeq \displaystyle \bigoplus_{0 \leq j \leq d} \bR \cH om_{X'}(S_{X'}(\cN^{j}_{X'})[-j], \omega^{\bullet}_{X'}).$$
To show $\cN^{j}_{X'}$ has RS, it suffices to show $S_{X'}(\cN^{j}_{X'})$ is Cohen-Macaulay of pure dimension $m$ by Lemma \ref{lemma:SCM}. 
From the above formula, we obtain
$$ R^i f_{\ast} Q_{X}(\cN) \simeq \displaystyle \bigoplus_{0 \leq j \leq d} \cE xt^{i+j}_{X'}(S_{X'}(\cN^{j}_{X'}), \omega^{\bullet}_{X'}).$$
By \cite[\href{https://stacks.math.columbia.edu/tag/0ECM}{Tag 0ECM}]{stacks-project} (or by locally embedding  $X'$ into a smooth
variety and applying \cite[Lemma 6.5]{Hart2}), we see that the dimension of support of  $\cE xt^{i+j}_{X'}(S_{X'}(\cN^{j}_{X'}), \omega^{\bullet}_{X'})$ is less than $m$ if $i+j>-m$. Hence if $R^{i}f_{\ast}Q_{X}(\cN)$ is torsion free for every $i$, then 
$$\cE xt^{i+j}_{X'}S_{X'}(\cN^{j}_{X'}), \omega^{\bullet}_{X'})= 0 \quad \text{unless $i +j = -m$}.$$ 
This is equivalent to $S_{X'}(\cN^{j}_{X'})_{p}$ being Cohen-Macaulay of pure dimension $m$.\\ 

$(3) \Rightarrow (2):$ If $\cN^{j}_{X'}$ has RS for every $j$, then 
$$ \bR f_{\ast} Q_{X}(\cN) \simeq \displaystyle \bigoplus_{0 \leq j \leq d} \bR \cH om_{X'}(S_{X'}(\cN^{j}_{X'})[-j], \omega^{\bullet}_{X'}) \simeq \displaystyle \bigoplus_{0 \leq j \leq d} \bR \cH om_{X'}(S_{X'}(\cN^{j}_{X'}), \omega^{\bullet}_{X'})[j]$$
$$\simeq \displaystyle \bigoplus_{0 \leq j \leq d} Q_{X'}(\cN^{j}_{X'})[j] \simeq \displaystyle \bigoplus_{0 \leq j \leq d} H^{-m}(Q_{X'}(\cN^{j}_{X'}))[j + m].$$
Therefore, $R^{-m-j}f_{\ast}Q_{X}(\cN) \cong  H^{-m}(Q_{X'}(\cN^{j}_{X'}))$ and
$$\bR f_{\ast} Q_{X}(\cN) \simeq \bigoplus_{0 \leq j \leq d} H^{-m}(Q_{X'}(\cN^{j}_{X'}))[j + m] \simeq \bigoplus_{0 \leq j \leq d} R^{-m-j}f_{\ast}Q_{X}(\cN) [m +j].$$ \\

$(2) \Rightarrow (1)$: By duality, and Theorem  \ref{thm:saito3}, we have the quasi-isomorphisms
$$ \bigoplus_{0 \leq j \leq d}S_{X'}(\cN^{j}_{X'})[-j] \simeq \bR f_{\ast} S_{X}(\cN) \simeq \cD(\bR f_{\ast}Q_{X}(\cN)) \simeq \bigoplus_{i} \cD(R^{i}f_{\ast}Q_{X}(\cN)[-i]).$$
When we take the $j^{th}$-cohomology on each side, we have the isomorphism
$$S_{X'}(\cN^{j}_{X'}) \cong \bigoplus_{i} \cE xt^{i +j}_{X'}(R^{i}f_{\ast}Q_{X}(\cN), \omega^{\bullet}_{X'}).$$
Since $S_{X'}(\cN^{j}_{X'})$ is torsion free, we may use the same argument from $(1) \Rightarrow (3)$ to prove
$$\cE xt^{i+j}_{X'}(R^{i}f_{\ast}Q_{X}(\cN), \omega^{\bullet}_{X'}) = 0 \quad \text{unless $i +j = -m$}.$$ 
Which implies $R^{i}f_{\ast}Q_{X}(\cN)$ is Cohen-Macaulay of pure dimension $m$, and therefore torsion-free.  

\end{proof}

\begin{rmk}
In the proof of $(3) \Rightarrow (2)$ it was shown there is a quasi-isomorphism
$$\bR f_{\ast} Q_{X}(\cN) \simeq \displaystyle \bigoplus_{0 \leq j \leq d} Q_{X'}(\cN^{j}_{X'})[j]. $$
But, in the statement of Theorem \ref{thm:saito3}, we have 
$$\bR f_{\ast} Q_{X}(\cN) \simeq \displaystyle \bigoplus_{q'(\cN^{j}_{X'}) = q'(\cN) } Q_{X'}(\cN_{X'}^{j})[-j].$$
It can be shown that $q'(\cN^{j}_{X'}) = q'(\cN)$ whenever $-d \leq j \leq 0$. Therefore,
$$\bR f_{\ast} Q_{X}(\cN) \simeq \displaystyle \bigoplus_{0 \leq j \leq d } Q_{X'}(\cN_{X'}^{-j})[j].$$
So, these two statements appear to be contradictory to each other. However, $f:X \rightarrow X'$ is a projective morphism and there is an isomorpism
$$ \ell^{j}: H^{-j}f_{+}\cN \rightarrow H^{j}f_{+}\cN (j),$$
where $\ell$ is the first Chern class of an $f$-ample line bundle \cite{saito1}. Hence there is an isomorphism $\cN^{-j}_{X'} \cong \cN^{j}_{X'}(j)$. Which then induces a quasi-isomorphism
$$Q_{X'}(\cN_{X'}^{j}) \simeq Q_{X'}(\cN_{X'}^{-j}).$$
So the two statements above agree.

\end{rmk}





\begin{cor}
 Let $f: X \rightarrow X'$ be a surjective projective map between irreducible algebraic varieties. Suppose that $X$ is smooth of dimension $n$ and  $\omega_X= f^*L$, where $L$ is a line bundle on $X'$.
 If  $\cN^{j}_{X'}$ is the largest direct factor of $H^{j}f_{+}\Q_X^H$ with strict support $X'$, then $\cN^{j}_{X'}$ has RS for all $j$.
\end{cor}

\begin{proof}
 By Koll\'ar's theorem \cite[Theorem 2.1]{kollar1} and the projection formula, $R^{i}f_{\ast}Q(\Q_X^H[n]) =R^{i}f_{*}\cO_{X}[n]= R^{i+n}f_*\omega_X \otimes L^{-1}$ is torsion free.
 So the corollary follows from the previous theorem.
\end{proof}

\section{Invariant theory}

In this section, we summarize some facts from geometric invariant theory needed below. The basic reference is \cite{mumford}.
Let $G$ be a complex linear algebraic group acting algebraically on a complex algebraic variety $X$ in the sense that 
 the action is defined by  a morphism
$\alpha:G\times X\to X$  of varieties.  
We will say that $X$ is a $G$-variety. A morphism of $G$-varieties is simply an equivariant morphism.
Given a $G$-variety,
the categorical quotient is a variety $Y$ with a morphism $\pi:X\to Y$
such that 
$$
\xymatrix{
 G\times X\ar@{=>}[r]^{\alpha}_{p_2} & X\ar[r]^{\pi} & Y
}
$$
is coequalizer, where $p_2$ is the second projection. If the categorical quotient exists, then it  is characterized up to isomorphism by
this property and we denote it by $X//G$.  The symbol $X/G$ is reserved for the set theoretic quotient, or orbit space, which might be different.
The action is effective if the map $G\to Aut(X)$ is injective. Effectivity can always be achieved by replacing $G$ by its image $\bar G$. 
The quotient $X//G= X//\bar G$ is unaffected.

The categorical quotient exists when $X= \Spec R$ is  an affine variety and $G$ is reductive.
In this case
$$X//G= \Spec R^G$$
Suppose  $X$ is a quasiprojective $G$-variety.  Then it possible to construct a categorical quotient of a large open set which depends on 
an equivariant locally closed embedding $X\subset \PP^N$. The embedding is determined by the line bundle
 $L=\OO_X(1)$, and $G$ will act on the space of sections of $L$ and its powers.
 A point $\bar x\in X$ is semistable (with respect to $L$) if the closure of the orbit of a point
$x\in \A^{N+1}-\{0\}$, lying over $\bar x$, does not contain $0$.  
The set of semistable points is denoted by $X(L)^{ss}$, or $X^{ss}$ if $L$ is understood or  unimportant.
This forms a nonempty $G$-invariant open set. Mumford \cite{mumford} shows that 
$$Y=  \im \left[\pi:X(L)^{ss}\to \Proj \left(\bigoplus_{i=0}^\infty H^0(\overline{X}, L^{\otimes i}) \right) ^G\right]$$
is the quotient $ X(L)^{ss}//G$, where $\overline{X}\subset \PP^N$ is the closure and $\pi$ is induced by the inclusion of projective coordinate rings.
 The quotient $\pi:X^{ss}\to Y$ is a so called good quotient, which means that  $\pi$ is affine, disjoint $G$-invariant closed
 sets map to disjoint closed sets, and $\OO_Y= (\pi_*\OO_{X^{ss}})^G$.  We note that the affine case is subsumed by the 
second construction because:

\begin{lemma}
 When $X$ is affine, we can choose
$L$ so that $X(L)^{ss}=X$. 
\end{lemma}

\begin{proof}
 We can choose an equvariant embedding $X\subset V$ into a rational representation. Let $\C$ denote the trivial representation.
 Then the closure of $X$ in $W=V\oplus \C$ does not contain $0$. Therefore the lemma holds for $L$ corresponding to 
 the embedding $X\subset \PP(W^*)$.
\end{proof}

We say that $Y$ is a geometric quotient of $X$ if it is a good quotient, and every point of $X$ has finite stabilizers and all orbits
are closed. The last condition implies  that $Y=X/G$ is the orbit space as a set.
A point $x\in X$ is stable if it has   a finite stabilizer and a closed orbit.   The set $X^s\subset X^{ss}$ of stable points
is an open, possibly empty, $G$-invariant subset.  The quotient $Y^s=X^s/G$ is a geometric quotient.
  Luna's theorem \cite[p 97]{luna} (see also \cite[appendix 1D]{mumford}) gives a good local description of geometric quotients.

\begin{thm}\cite{luna}\label{thm:luna}
 If $G$ is a reductive group and $\pi:X\to X/G=Y$ is a geometric quotient, then for each $x\in X$, there exists a finite group $H\subset G$,
 an $H$-subvariety $ S\subset X$, and a commutative diagram
 $$
\xymatrix{
 G\times S\ar[r]\ar[d] & (G\times S)/H\ar[r]^>>>>>{p}\ar[d] & X\ar[d]^{\pi} \\ 
 S\ar[r] & S/H\ar[r]^{p} & Y
}
 $$
 such that maps labeled $p$ are \'etale and they map onto Zarski open neighbourhoods of $x$ and $\pi(x)$.
 If $X$ is smooth, then $S$ can be chosen smooth.
\end{thm}




\section{Equivariant Hodge modules}

Suppose that $X$ is a smooth $G$-variety with quotient $\pi:X\to Y=X//G.$ Let $\F'$ be a sheaf on $X//G$, and let $\F=\pi^*\F'$.
Then since $\pi\circ p_2=\pi\circ \alpha$, it follows that $p_2^*\F\cong \alpha^*\F$.  Furthermore, $\pi^*id$ is
an isomorphism $\phi:p_2^*\F\cong \alpha^*\F$ which
satisfies the cocycle condition:
$$s^{*}\phi = id \quad \text{and} \quad b^{*}\phi \circ q^{*}\phi = m^{*} \phi,$$
where the maps $s, b, q,$ and $m$ are given below.
\begin{center}
$\begin{cases}
s: X \rightarrow G \times X & s(x) = (e, x) \\
m: G \times G \times X \rightarrow G \times X & m(g,h,x) = (gh,x)\\
b: G \times G \times X \rightarrow G \times X & b(g,h,x) = (g, \alpha(h,x))\\
q:G \times G \times X \rightarrow G \times X & q(g,h,x) = (h,x)\\
\end{cases}$
\end{center}
See also \cite[Defn 9.10.3]{htt} or \cite{achar}. A $G$-equivariant sheaf is a pair $(\F,\phi)$, where $\phi$ is an isomorphism as above satisfying
the cocycle condition.
 Equivariance implies that $g^*\F\cong \F$ for all $g\in G$, but  it is stronger than this.
 It implies that $G$ will act on the cohomology groups of $\F$. This fact was already used implicitly in the previous section.
 A morphism of equivariant sheaves is required to be compatible
with the $\phi$'s. Let $Coh_G(X)$ be the category of coherent equivariant sheaves.

\begin{prop}\label{prop:free}
  When  $G$ acts freely  on $X$ in the sense that $X\to X/G$ is a locally trivial principal bundle in the \'etale topology,
  $\pi^*$ gives an equivalence of categories $Coh(X/G)\approx Coh_G(X)$. The inverse sends $\F\mapsto
  (\pi_*\F)^G$.
  \end{prop}
   
   We omit the details, but this is a standard consequence of faithfully flat descent. 
   
   \begin{defn}
   In general, for possibly nonfree actions, the functor $\pi_*^G:Coh_G(X)\to Coh(X//G)$ defined by   $\F\mapsto \pi_{*}^{G}\F:=(\pi_*\F)^G$ is right adjoint to $\pi^*$.  
   \end{defn}

\begin{lemma}\label{lemma:piGsplits}
The functors $\pi_*$  and  $\pi^G_*$ are exact.
 When $G$ is reductive and $\F\in Coh_G(X)$, $\pi^G_*\F$ is a direct summand of $ \pi_*\F$.
\end{lemma}

\begin{proof}
The first statement follows from the fact that the  map $\pi$ is affine.
 For any open set $U\subset Y$, reductivity of $G$ implies that $\F(\pi^{-1}U)^G\subset \F(\pi^{-1}U)$
 is a direct summand.
\end{proof}
  
   If $X=\Spec R$, and $V$ is a rational representation of $G$,
   then $M=V\otimes_\C R$, with diagonal action, is naturally    
   an object of $Coh_G(X)$. In fact, the set of  objects of this form generates this category.
   We have that $\pi_*^G\tilde M= \widetilde{M^G}$.

 \begin{defn}\cite[Defn. 5.1]{achar}
 A $G$-equivariant Hodge module on $X$ is a pair $(\cM, \beta),$ where $\cM \in MHM(X)$, and $\beta: p_{2}^{*}\cM \rightarrow \alpha^* \cM$ is an isomorphism satisfying the cocycle condition.
\end{defn}
 
 Further information about equivariant (mixed) Hodge modules can be found in Achar's notes  \cite{achar}.
 Let $HM_{?,G}(X, ??)$ denote the category of $G$-equivariant Hodge modules (with strict support ? and weight  ??).
This category can be described as follows:
 
\begin{prop}\label{prop:GVHS}
If $Z\subseteq X$ is a $G$-invariant closed subvariety,
a variation of Hodge structure  of weight $w$ on a smooth $G$-invariant Zariski open set $U\subseteq X$ 
with  an isomorphism
 $\phi:p_2^*V\cong \alpha^*V$ satisfying the cocycle identity extends to $G$-equvariant Hodge module
 of weight $w+\dim Z$. In fact,
 there is an equivalence of categories
 $$HM_{Z,G}(X, w)\approx 2\text{-}\varinjlim_{U\subseteq Z}VHS_G(U,w-\dim Z)$$
 \end{prop}

\begin{proof}
 This is an immediate consequence of  \cite[Theorem 3.21]{saito2}.
\end{proof}

\begin{cor}\label{cor:Ggen}
 If $j:U\subset X$ is a nonempty $G$-invariant open set, restriction $j^*$ gives an equivalence
 $$HM_{X,G}(X, w)\approx HM_{U,G}(U,w)$$
 with inverse $j_{!*}$.
\end{cor}
    
 Let $d=\dim G$, and set $\pi^{t} \cN :=\pi^*\cN[d]$.  Note that $\pi^{t}$ is compatible
 with the usual pull back of variations of Hodge structure, i.e. when $\cN$ corresponds to $V\in VHS(U)$, with $U\subset X$ open,
$\pi^{t} \cN$ corresponds to $\pi^*V$.
 If $\cN\in HM(X//G)$, 
  then $(\pi^t \cN, \pi^t id)$ (which we simply  denote by $\pi^t \cN$)
gives   an object of $HM_G(X)$. 
  Proposition \ref{prop:free} works for Hodge modules.
 
 \begin{prop}\label{prop:freeHM}
  When $G$ acts freely on  $X$, $\pi^t$  gives an equivalence of categories between $HM(X/G)$ and $HM_{G}(X)$.
  More generally, if $G$ has a normal subgroup $H$ that acts freely on $X$, then $\pi^t$  gives an equivalence 
  $HM_{G/H}(X/H)\approx HM_{G}(X)$.
  \end{prop}

\begin{proof}
 See \cite[\S 6]{achar}. Note that $\pi^t$ is denoted by $\pi^\dagger$ in [loc. cit.].
\end{proof}

More generally, we will prove the following proposition.

\begin{prop}\label{EquivProp} Suppose that $G$ is a reductive group that acts effectively on $X$ such that $X^s\not=\emptyset$. Let $Y = X//G$.
Then there is a functor 
$$\pi^{\dagger}: HM_{Y}(Y,w) \rightarrow HM_{X,G}(X,w+d),$$
which gives an  equivalence of categories, and which 
 agrees with  $\pi^t$ when $G$ acts freely.
 \end{prop}

To prove the proposition, we first need to prove two lemmas.

\begin{lemma}\label{lemma1}
If $f: S \rightarrow T$ is a surjective finite map between irreducible algebraic varieties, then there is an equivalence of categories between $HM_{S}(S,w)$ and $HM_{T}(T,w).$ The equivalence  are given by the functors

$$f^\#: HM_{T}(T,w) \rightarrow HM_{S}(S,w)$$ 
$$\cN \mapsto \cM = \text{maximal summand of $Gr^{W}_{w}(\cH^{0}(f^{*}\cN))$  with strict support on $S$}$$ 
and 
$$f_\#: HM_{S}(S,w) \rightarrow HM_{T}(T,w)$$ 
$$\cM' \mapsto \cN' = \text{maximal summand of $\cH^{0}(f_{+}\cM') $ with strict support on $T$.}$$
\end{lemma}

\begin{proof}
First consider $\cN \in HM_{T}(T,w).$ Since the fibers of $f$ are zero dimensional, $\cH^{i}(f^{*}\cN) = 0$ for $i > 0$ \cite[ Prop. 8.1.40]{htt}. By Saito's theory of weights \cite[Prop. 2.26]{saito2}, $Gr^{W}_{w+1}(\cH^{0}(f^{*}\cN) )= 0$. We have a functor 
$$f^\#: HM_{T}(T,w) \rightarrow HM_{S}(S,w)$$ 
$$\cN \mapsto \cM = \text{maximal summand of $Gr^{W}_{w}(\cH^{0}(f^{*}\cN))$  with strict support on $S$.}$$ 
Now suppose that $\cM' \in HM_{S}(S,w).$ Since the map $f$ is finite, $f_{+}\cM' \simeq \cH^{0}(f_{+}\cM').$  There is a functor 
$$f_\#: HM_{S}(S,w) \rightarrow HM_{T}(T,w)$$ 
$$\cM' \mapsto \cN' = \text{maximal summand of $\cH^{0}(f_{+}\cM') $ with strict support on $T$.}$$
We will show that $f_\#\circ f^\#$ is isomorphic to the identity functor on $HM_{T}(T,w)$ and $f^\# \circ f_\#$ is isomorphic to the identity functor on $HM_{S}(S,w).$

By construction, there is a natural map $\cN \rightarrow (f_\# \circ f^\#) (N)$. Using \cite[Thm. 3.21]{saito2} and possibly restricting to a smooth Zariski-open subset of $T,$ we may assume $T$ is smooth, $\cN$ is a variation of Hodge structure, and the map $f$ is \'etale of degree $\ell.$ Let $\cL$ denote the underlying locally constant sheaf of $\cN$. The adjunction map
$$ad j_{*}: \cL \rightarrow \bR f_{*} f^{-1}\cL \simeq R^{0}f_{*} f^{-1}\cL$$
realizes $\cL$ as a direct summand of $R^{0}f_{*} f^{-1}\cL$ because the composition $\displaystyle \frac{1}{\ell} \cdot ad j_{!} \circ ad j_{*}$ is the identity map on $\cL,$ where $ad j_{!}$ is given by the map 
$$R^{0}f_{*} f^{-1}\cL \simeq \bR f_{!} f^{!}\cL \rightarrow \cL.$$
Therefore, the natural map $\cN \rightarrow f_{+}f^{*}\cN$ splits and $\cN$ is a direct summand of $f_{+}f^{*}\cN.$ Hence the natural map $\cN \rightarrow (f_\#\circ f^\#)(\cN)$ is an isomorphism.

For the other direction, there is a natural map $ (f^\# \circ f_\#)(\cM') \rightarrow \cM'.$  Again using \cite[Thm. 3.21]{saito2}, after possibly restricting to smooth Zariski-open subsets of $S$ and $T,$ we may assume $\cM'$ and $(f^\# \circ f_\#)(\cM')$ are variations of Hodge structure, and the map $f:S \rightarrow T$ is \'etale. But then by construction, the map $(f^\# \circ f_\#)(\cM') \rightarrow \cM'$ is a local isomorphism, and hence the map is an isomorphism.

\end{proof}



\begin{lemma}\label{lemma2}
    Suppose $f: S \rightarrow T$ is a surjective, finite map between smooth, irreducible $G$-varieties, and the map $f$ is equivariant. Then there is an equivalence of categories  $f^\#:HM_{S,G}(S,w)\approx HM_{T,G}(T,w).$
\end{lemma}

\begin{proof}
    By the previous lemma, we already know there is an equivalence of categories between $HM_{S}(S,w)$ and $HM_{T}(T,w)$. We have to show this equivalence of categories is compatible with the
    data defining equivariant Hodge modules.  Using the notation of the previous proof, 
    we will show that the compatibility for the  functor $f^\#:HM_{T}(T,w) \rightarrow HM_{S}(S,w)$. The argument for the  inverse functor 
    $f_\#$ is similar.
    
        Since the map $f$ is equivariant, we have the following commutative diagrams
    $$\bT
    G \times S \arrow[r, "\alpha_{S}"] \arrow[d, "(id; f)" ] & S \arrow[d, "f"] \\
    G \times T \arrow[r, "\alpha_{T}"] & T \eT
    \quad \quad 
    \bT
    G \times S \arrow[r, "p_{S}"] \arrow[d, "(id ; f)" ] & S \arrow[d, "f"] \\
    G \times T \arrow[r, "p_{T}"] & T, \eT$$
   where $p$ is the projection map and $\alpha$ is the action map. Note that the map $(id; f): G \times S \rightarrow G \times T$ is finite, and hence there is an equivalence of categories between $HM_{G \times S}(G \times S)$ and $HM_{G \times T}(G \times T)$. If $\cN \in HM_{T,G}(T,w),$ then there is an isomorphism $\beta_{T}: p_{T}^{*}\cN \rightarrow \alpha^{*}_{T}\cN$ satisfying the cocycle condition. Therefore, there is a quasi-isomorphism
   $$ p_{S}^{*} f^{*} \cN \simeq (id; f)^{*} p_{T}^{*} \cN \rightarrow (id; f)^{*} \alpha_{T}^{*} \cN \simeq \alpha_{S}^{*} f^{*} \cN.$$
   The projection and action maps are smooth. Hence there are quasi-isomorphisms 
   $$p^{*}_{S}\cN \simeq \cH^{d}(p^{*}_{S}\cN)[-d] \quad \text{and} \quad \alpha^{*}_{S}\cN \simeq \cH^{d}(\alpha^{*}_{S}\cN)[-d],$$
   where $d= \dim G$ \cite[Lemma 2.25]{saito2}. Therefore, by the use of spectral sequences, we have an isomorphism
   $$\cH^{d}(p_{S}^{*}(\cH^{0}(f^{*}\cN))) \cong  \cH^{d}( p_{S}^{*} f^{*} \cN) \rightarrow \cH^{d}(\alpha_{S}^{*}f^{*}\cN) \cong \cH^{d}(\alpha_{S}^{*}(\cH^{0}(f^{*}\cN))).$$
   Also, by \cite[Lemma 2.25]{saito2},
   $$Gr^{W}_{w +d}\cH^{d}(p_{S}^{*}(\cH^{0}(f^{*}\cN))) \cong \cH^{d}(p_{S}^{*}(Gr^{W}_{w}\cH^{0}(f^{*}\cN))) $$
   $$ Gr^{W}_{w+d}\cH^{d}(\alpha_{S}^{*}(\cH^{0}(f^{*}\cN))) \cong \cH^{d}(\alpha_{S}^{*}(Gr^{W}_{w}\cH^{0}(f^{*}\cN))).$$
   Hence there is an isomorphism $\beta_{S}:p_{S}^{*}\cM \rightarrow \alpha_{S}^{*}\cM.$ The cocycle condition can be shown in a similar matter. Therefore, $\cM$ is an equivariant Hodge module.
\end{proof}



\begin{proof}[Proof of Proposition \ref{EquivProp}]
 Since $X^{s} \neq \emptyset,$ by Theorem \ref{thm:luna},  we have a diagram 
  $$
\xymatrix{
 G\times S\ar[r]^{r}\ar[d]^{p} & (G\times S)/H\ar[r]^{f}\ar[d] & X^s\ar[d]^{\pi}\\ 
 S\ar[r]^{q} & S/H\ar[r]^{h} & Y^s
}
 $$
satisfying the conditions stated in that theorem. There is also an equivalence of categories $$HM_{X^{s},G}(X^{s},w+d) \approx HM_{(G \times S)/H, G}((G\times S)/H, w +d)$$
by Lemma \ref{lemma2} because the map $(G \times S)/H \rightarrow X^{s}$ is $G$-equivariant. Note that the map $G \times S \rightarrow (G \times S)/H$ is $G$-equivariant because the $G$-action on  $(G \times S)/H$ is the natural action
$$g \cdot [(g',s)] =[(gg',s)]  \quad \text{for $g,g' \in G$ and $s \in S.$}$$ 
So, by applying Lemma \ref{lemma2}, we have an equivalence of categories 
$$HM_{G \times S, G}(G\times S, w+d) \approx  HM_{(G \times S)/H, G}((G\times S)/H, w +d).$$
Using a similar argument with Lemma \ref{lemma1}, there is an equivalence of categories 
$$HM_{Y}(Y,w) \approx HM_{S}(S,w).$$
With Proposition \ref{prop:freeHM} we   obtain the diagram
$$
\xymatrix{
 HM_{G\times S, G}(G\times S, w +d) &   & HM_{X^s,G}(X^s, w+d) \ar_{\approx}[ll]\\ 
 HM_{S}(S, w)\ar_{\approx}[u] & HM_{S/H}(S/H, w)\ar_{\approx}[l] & HM_{Y^s}(Y^s,w).\ar_{\approx}[l]\ar[u]
}
$$

 By Corollary \ref{cor:Ggen}, there are equivalences of categories 
$$j^*:HM_{X,G}(X,w+d) \approx HM_{X^{s},G}(X^{s},w+d),$$
$$k^*:HM_{Y}(Y,w) \approx HM_{Y^{s}}(Y^{s},w),$$
where $j:X^s\to X$ and $k: Y^s\to Y$ are the inclusions.
Then
\begin{equation}\label{eq:pidagger}
\pi^\dagger = j_{!*} f_\#r_\# p^t  q^\# h^\# k^* 
\end{equation}
gives the desired equivalence. One can check that
 $\pi^{\dagger}=\pi^t$  when $G$ acts freely on $X$ and $\pi=\pi^\#$ when $G$ is finite.
\end{proof}

\begin{rmk}
     The formula \eqref{eq:pidagger} is not easy to use directly.
     However, it is not hard to see that the functor $\pi^{\dagger}: HM_{Y}(Y,w) \rightarrow HM_{X,G}(X,w +d)$  is given by:
    $$\cN \mapsto \cM =\text{ the maximal summand of $Gr^{W}_{w +d}\cH^{d}(\pi^{*}\cN)$ with strict support on $X$.}$$
  \end{rmk}
   
 Here is a basic example.   
    
\begin{lemma}\label{lemma:piIC}
With the same assumptions as in Proposition \ref{EquivProp},  $\pi^\dagger IC_Y^H= IC_X^H$.
\end{lemma}

\begin{proof}
Choosing smooth Zariski open sets $U\subset X$ and $V\subset Y$ such that $\pi(U)\subset V$, the pullback of the constant variation of Hodge structure
$\pi^*\Q_V^H= \Q_U^H$. Since $IC_Y^H$ and $IC_X^H$ are the unique extensions of $\Q_V^H[\dim Y]$ and $ \Q_U^H[\dim X]$ to  Hodge modules with strict support on $Y$ and $X$,
we must have $\pi^\dagger IC_Y^H= IC_X^H$.
\end{proof}

    We will use $\pi^{G}_{+}: HM_{X,G}(X, w+d) \rightarrow HM_{Y}(Y,w)$ to denote the inverse functor of $\pi^{\dagger}.$ The next proposition gives an explicit description of $\pi^{G}_{+}.$

    \begin{prop}
        If $\cM \in HM_{X,G}(X, w +d),$ then
        $$\pi^{G}_{+}(\cM) = \text{ the maximal summand of $Gr^{W}_{w}\cH^{-d}(\pi_{+}\cM)$ with strict support on $Y$.}$$
    \end{prop}
    \begin{proof}
        If $G$ is finite, then the proposition follows from Lemma \ref{lemma1}. So, we may assume $\dim G \geq 1$ and $G$ is connected. After possibly restricting to open subsets of $X$ and $Y$, we may assume $\pi:X \rightarrow Y$ is a geometric quotient. By Luna's theorem \ref{thm:luna},  we have the diagram 
  $$
\xymatrix{
 G\times S\ar[r]^{r}\ar[d] & (G\times S)/H\ar[r]\ar[d] & X^s\ar[d]^{\pi}\\ 
 S\ar[r]^{q} & S/H\ar[r] & Y^s
}
 $$
satisfying the assumptions stated in that theorem. After possibly restricting further to open subsets of $X$ and $Y,$ we may assume the maps
$$S/H \rightarrow Y \quad \text{and} \quad  (G \times S)/H \rightarrow X$$
are surjective. Hence we may asssume the maps $G \times S \rightarrow X$ and $S \rightarrow Y$ are surjective and finite. Therefore, by Lemma \ref{lemma1} and Lemma \ref{lemma2}, we may reduce to the case when $X = G \times Y$ and $\pi:X \rightarrow Y$ is the natural projection. Let $s:Y \rightarrow X$ be the natural map given by 
$$y \mapsto (e, y) \quad \text{ $e \in G$ is the identity element.}$$
We have $\pi \circ s= id_{Y}$, and $s^{\dagger}: HM_{X,G}(X, w+d) \rightarrow HM_{Y}(Y,w)$ is the inverse functor for the equivalence of categories. For $\cN \in HM_{Y}(Y,w)$ it was shown in the proof of \cite[Lemma 2.27]{saito2} that the map
$$\cH^{-d}(\pi_{+}(\cH^{d}\pi^{*}\cN)) \rightarrow \cH^{0}s^{\dagger}\pi^{\dagger}\cN = \cN$$
is an isomorphism.
    \end{proof}
    
    \begin{cor}\label{cor:proj}
        If  $\cM \in HM_{X,G}(X, w +d),$ then
        $$\pi^{G}_{+}(\cM) = \text{the maximal summand of $\cH^{-d}(\Bar{\pi}_{+}\Bar{\cM})$ with strict support on $Y$,}$$
        where $\Bar{\pi}: \Bar{X} \rightarrow Y$ is any compactification of $\pi,$ and $\Bar{\cM} \in HM_{\Bar{X}}(\Bar{X},w+d)$ is the unique extension of $\cM.$ 
    \end{cor}

    \begin{proof}
        Compactify the morphism $\pi$ to a projective morphism $\bar{\pi}:\Bar{X}\rightarrow Y$ such that the diagram commutes
    $$\bT X \arrow[r, hook, "j"] \arrow[dr, "\pi"] & \Bar{X} \arrow[d, "\Bar{\pi}"] \\
    & Y. \eT$$
    There is a natural map
    $$\cH^{-d}(\Bar{\pi}_{+}\Bar{\cM}) \rightarrow Gr^{W}_{w}\cH^{-d}(\pi_{+}\cM).$$
    Furthermore, this map is surjective by the weight spectral sequence \cite[Prop. 2.15]{saito2}. Which implies $\pi^{G}_{+}(\cM)$ is a direct summand of $\cH^{-d}(\Bar{\pi}_{+}\Bar{\cM})$.
    \end{proof}

    \begin{lemma}
    If $f:S \rightarrow T$ is a surjective finite map between irreducible algebraic varieties, then $\cM \in HM_{S}(S,w)$ has RS if and only if $\cN \in HM_{T}(T,w)$ has RS under the equivalence of categories.
\end{lemma}

\begin{proof}
    The morphism $f:S \rightarrow$ is projective since it is finite. By Theorem \ref{thm:saito3}, and using $f_{*}$ is exact, 
    $$\bR f_{*}Q_{S}(\cM) \simeq f_{*}Q_{S}(\cM) \simeq Q_{T}(\cN).$$
    So, if $\cM$ has RS, then it is clear that $\cN$ must also have RS. Similarly, since $f:S \rightarrow T$ is affine and $f_{*}$ is exact, if $\cN$ has RS then $\cM$ must also have RS.

\end{proof}

\begin{lemma}\label{lemma:piG}
 Let $G$ be a reductive group of dimension $d$ and $\pi:X\to X/G=Y$ a geometric quotient of a smooth variety. If $\cM\in HM_{X,G}(X)$,
  then $\cM$ has rational singularities if and only if $\pi^{G}_{+}\cM$ has rational singularities. In particular, if $\cM$ or $\pi^{G}_{+}\cM$ has rational singularities, then $Q(\pi_+^G \cM)[d] \cong \pi_*^G Q( \cM)$ 
and $ S(\pi_+^G \cM)\cong \pi_*^G S( \cM)$. 
\end{lemma}

\begin{proof}
Recall, by Theorem \ref{thm:luna}, for $x \in X$ and $\pi(x) \in Y$ there exists open neighborhoods $x \in U \subset X$ and $\pi(x) \in V \subset Y$ such that we have the commutative diagram
$$
\xymatrix{
 G\times S\ar[r]\ar[d]^{\Pi} & (G\times S)/H\ar[r]^>>>>>{p}\ar[d] & U \ar[r] \ar[d]^{\pi\vert_{U}} & X\ar[d]^{\pi} \\ 
 S\ar[r] & S/H\ar[r]^{p} & V \ar[r] &Y.
}
 $$
The variety $S$ is smooth, and the maps $p$ are \'etale and surjective. Since the condition of rational singularities is local, we may replace $X$ with $U$ and $Y$ with $V$. Hence we may assume the induced maps
$$\alpha: G \times S \rightarrow X \quad \text{ and} \quad \beta: S \rightarrow Y$$
are finite and surjective. Let $\cN = \pi^{G}_{+}\cM.$ Recall from Lemma \ref{lemma1} there exists an equivalence of categories between $HM_{Y}(Y,w)$ and $HM_{S}(S,w).$ Let $\cN' \in HM_{S}(S)$ be the Hodge module that corresponds to $\cN.$ Now, because the map  $S \rightarrow Y$ is finite, $\cN$ has RS if and only if $\cN'$ has RS. Similarly, let $\cM' \in HM_{G\times S, G}(G \times S)$ be the Hodge module that corresponds to $\cM$ from Lemma \ref{lemma2}. Note that $\cM$ has RS if and only $\cM'$ has RS. By construction, we have
$$ \cM' = \Pi^{*}\cN'[d].$$
The explicit formulas in \cite[\S30]{schnell} show that 
$$Q_{G \times S}(\cM') = \Pi^*Q_{S}(\cN')[d]$$
and
$$S_{G \times S}(\cM') = \Pi^*S_{S}(\cN').$$
Since $\Pi:G \times S \rightarrow S$ is faithfully flat, we see that $\cM'$ has RS if and only if $\cN'$ has RS. Hence $\cM$ has RS if and only if $\cN$ has RS. 

When $\cM'$ or $\cN'$ has RS, we have
\begin{equation}\label{eq:PiG}
\begin{split}
 Q_{S}(\cN')[d] &\cong \bR \Pi^{G}_{*}Q_{G \times S}(\cM')= \Pi^{G}_{*}Q_{G \times S}(\cM')\\
 S_{S}(\cN') &\cong \bR \Pi^G_*S_{G \times S}(\cM') = \Pi^G_*S_{G \times S}(\cM')
\end{split}
\end{equation}
because  $G$ acts freely on $S$  and we may apply Proposition \ref{prop:free}.  The second set of equalities, which involves a slight abuse of notation,
stems from the fact that $\Pi^G_*$ is exact on the category of coherent sheaves by Lemma \ref{lemma:piGsplits}.
Now it is well known that if $S$ is an affine regular $\C$-domain, then the ring of $G$-invariants $S^{G}$ is a direct summand of $S$ as a $S^{G}$-module \cite{hr}. The  sheaves  (up to translation) $Q_{S}(\cN')$, $Q_{G \times S}(\cM')$, $S_{S}(\cN')$ and $S_{G \times S}(\cM')$ are locally free  by Corollary \ref{locallyfree}, because $S$ and $G \times S$ are smooth. By \eqref{eq:PiG}, $Q_{S}(\cN')[d]$  is  a direct summand of $\Pi_{*}Q_{G \times S}(\cM').$ Let
$$\Pi_{*}Q_{G \times S}(\cM') = \Pi^{G}_{*}Q_{G \times S}(\cM') \oplus A[\dim X] = Q_{S}(\cN')[d] \oplus A [\dim X]$$
for some coherent $\cO_{S}$-module $A$ by Lemma \ref{lemma:piGsplits}. Then
$$\pi_{*}Q_{X}(\cM) \cong \pi_{*}\alpha_{*}Q_{G \times S}(\cM') \cong  \beta_{*}\Pi_{*}Q_{G \times S}(\cM') \cong   \beta_{*}\Pi^{G}_{*}Q_{G \times S}(\cM') \oplus \beta_{*}A[\dim X]$$
$$\cong \beta_{*}Q_{S}(\cN')[d] \oplus \beta_{*}A[\dim X] \cong Q_{Y}(\cN)[d] \oplus \beta_{*}A[\dim X].$$
Therefore, when taking $G$-invariants, we obtain
$$Q_{Y}(\cN)[d] \cong  \pi^{G}_{*}Q_{X}(\cM).$$
By a similar argument,
$ S(\pi_+^G \cM) \cong \pi_*^G S( \cM)$. 
\end{proof}

\begin{thm}\label{thm:main}
     Let $G$ be a $d$-dimensional reductive algebraic group  with an effective action on a smooth affine variety $X$ such that $X^s\not=\emptyset$.  Let $\pi:X\to Y= X//G$ be the quotient map. 
     Suppose that  $\cM \in HM_{X,G}(X,w+d)$ has rational singularities, then $\cN = \pi^{G}_{+}\cM \in HM_{Y}(Y,w)$ has rational singularities, and  
     $$Q_{Y}(\cN)[d] \cong (\pi_{*}Q_{X}(\cM))^{G}$$
     $$S_{Y}(\cN) \cong (\pi_{*}S_{X}(\cM))^{G}.$$
\end{thm}

\begin{proof}
     By a theorem of Kirwan \cite[Prop 6.3]{kirwan}, we can find a nonsingular $G$-variety $\tilde X$ with  an ample $G$-equivariant line bundle $\tilde L$, and an equivariant map $p:\tilde X\to X$, such that $p(\tilde X(\tilde L)^{ss})=X$,  $\tilde X(\tilde L)^{ss}= \tilde X(\tilde L)^{s}$. We also have  that the induced map $q:\tilde Y= X^{ss}//G\to Y$ is  birational.  Furthermore,  the arguments of \cite[\S 3]{kirwan} show that $\tilde X$ can be chosen to dominate any given equivariant blow up of $X$. Therefore, the pull back of a given invariant closed set can be assumed to be a divisor with simple normal crossings. To simplify notation, let us replace $\tilde X$ by  $\tilde X^{ss}$. Then  we have a commutative diagram
$$
\xymatrix{
 \tilde X\ar[d]^{\tilde \pi}\ar[r]^{p} &X\ar[d]^{\pi}\\ 
  \tilde Y\ar[r]^{q} &Y.
}
 $$
 The map $p: \tX \rightarrow X$ is birational. So let $\tilde \cM \in HM_{\tX,G}(\tX,w+d)$ be the unique Hodge module that agrees with $\cM$ on some smooth open dense subset. We can assume that $\tilde \cM$ has singularities along a normal crossing divisor by blowing up $\tilde X$ further. So, by Lemma \ref{lemma:ncdrs},  $\tilde \cM$ has RS. Since the map $\tilde \pi: \tX \rightarrow \tY$ is a geometric quotient,  $\tilde \cN= \tilde \pi^G_+ \tilde \cM$ has RS by Lemma \ref{lemma:piG}. 
Since the maps $\pi$ and $\tilde \pi$ are affine, the functors $\pi_*$ and $\tilde \pi_*$ are exact on the category of coherent sheaves,
so we use the same symbols for their derived functors.
Since 
 $$Q_{\tY}(\tcN)[d] \cong \tilde \pi^{G}_{*}Q_{\tX}(\tcM)$$
 $$S_{\tY}(\tcN) \cong \tilde{\pi}^{G}_{*}S_{\tX}(\tcM)$$
 we deduce that
 $$\tilde \pi_{*}Q_{\tX}(\tcM) \cong Q_{\tY}(\tcN)[d] \oplus A[\dim X]$$
 $$ \tilde{\pi}_{*}S_{\tX}(\tcM) \cong S_{\tY}(\tcN) \oplus B$$
 for some $\cO_{\tY}$-modules $A$ and $B$ by Lemma \ref{lemma:piGsplits}.
  Therefore, we have
$$\pi_{*} Q_{X}(\cM) \simeq \pi_{*} \bR p_{*}Q_{\tX}(\tcM) \simeq \bR q_{*} \tilde \pi_{*}Q_{\tX}(\tcM)  $$

$$\simeq \bR q_{*}Q_{\tY}(\tcN)[d] \oplus \bR q_{*}A[\dim X] \simeq Q_{Y}(\cN)[d] \oplus \bR q_{*}A[\dim X]. $$
Similarly, 
$$\pi_{*}S_{X}(\cM) \simeq S_{Y}(\cN) \oplus \bR q_{*}B.$$
Therefore if $ \cM$ has RS, then we see that $\cN$ has RS and 
$$ Q_{Y}(\cN)[d] \cong \pi^{G}_{*}Q_{X}(\cM)$$
$$S_{Y}(\cN) \cong \pi^{G}_{*}S_{X}(\cM).$$
  \end{proof}

We recover a special case of Boutot's theorem (needed below) by a method   entirely different from the original proof.

\begin{cor}
Let $X,G$ and $Y$ be as in  Theorem \ref{thm:main}. The variety $Y$ has rational singularities. 

\end{cor}

\begin{proof}
    We note first of all that $Y$ is normal because the normalization of $Y$ would also be a categorical quotient of $X$ by $G$.
    Since $X$ is smooth, $\Q^{H}_{X}[\dim X] \in HM_{X,G}(X, \dim X)$ has rational singularities.
    By Lemma   \ref{lemma:piIC}, $\pi_*^G\Q^{H}_{X}[\dim X]=IC^{H}_{Y} $, and this has rational singularities by
    Theorem \ref{thm:main}.
     Therefore $Y$ has rational singularities by Corollary \ref{cor:rs}.
   
\end{proof}

The following should be  standard, but we indicate the proof for lack of a suitable reference.

\begin{lemma}\label{lemma:CMsummand}
A direct summand of a Cohen-Macaulay module over a noetherian ring is again Cohen-Macaulay.
\end{lemma}

\begin{proof}
 It suffices to prove this when the ring is local.
 If $N\subset M$ is a summand, then we have inequalities
 $$\dim M\ge \dim N\ge \depth N\ge \depth M.$$
 Therefore if $M$ is CM, this forces $N$ to be CM as well.
 We only prove the last inequality, since the remaining inequalities are elementary.
 If $k$ denotes the residue field, then by \cite[Theorem 16.7]{matsumura}
 $$\depth N = \inf \{i\mid Ext^i(k,N)\not=0\}\ge  \inf\{i\mid Ext^i(k,M)\not=0\}= \depth M$$
 because $Ext^i(k,N)$ is a summand of $Ext^i(k,M)$.
\end{proof}

\begin{cor}
Let $X=\Spec R$ be a smooth variety admitting  an action by a reductive group $G$ such that $X^s\not=\emptyset$.
 Let $V$ be  a rational representation of $G$. Then $(V\otimes_\C R)^G$  is Cohen-Macaulay  provided  that $V\otimes R$ is a direct summand 
 of the global sections of   either
  $Q_{X}(\cM)[-\dim X]$ or $S_X(\cM)$ for
 some $\cM\in HM_{X,G}(X)$ with RS.
\end{cor}

\begin{proof}
Set $Y= \Spec R^G = X//G$.
By the previous corollary $Y$ has rational singularities, therefore  it is CM \cite[chap I \S3]{toroidal}. 
  Thus $\omega_Y^\bullet=\omega_Y[\dim Y]$, where 
    $\omega_Y $ is the dualizing sheaf.   Let  $\cM\in HM_{X,G}(X)$ have RS and let  $\cN = \pi^{G}_{+}\cM$.
    Theorem \ref{thm:main} implies that $\cN$ has RS. Therefore $S_Y(\cN)$ is maximal Cohen-Macaulay by Lemma \ref{lemma:SCM}.
    So by Theorem \ref{thm:main}  and \cite[Satz 6.1]{herzog}
    $$\pi_*^G(Q_X(\cM)[-\dim X ])= \mathcal{H}om_{Y}(S_Y(\cN),\omega_Y)$$
    is CM. A direct summand of $Q_{X}(\cM)[-\dim X]$ or $S_X(\cM)$ is also CM by Lemma \ref{lemma:CMsummand}.
\end{proof}

Following the referee's request, we give  a simple  example where this corollary applies. 

\begin{ex}\label{ex:CM}
Let $G\subseteq GL_n(\C)$ be a reductive subgroup acting in the standard way on  $R= \C[x_1,\ldots, x_n]$. We suppose that
$X^s\not=\emptyset$.
In the special case, where $G= \C^*$, with $t\in G$ acting  by $x_i \mapsto t^{a_i} x_i$, $X^s\not=\emptyset$ if and only if 
  there exists $ j,k$ with $a_j> 0$ and $a_k< 0$. 
 Let $V= \C v$, with  $A\in G$ acting by $v\mapsto \det(A) v$. Then $(V\otimes R)^G$
 is CM because we can identify $V\otimes R\cong  \Omega^n_R \cong \Gamma(Q_X(\Q_X^H[n])[-n])$.
\end{ex}

\printbibliography[
heading=bibintoc,
title={References}
] 

\end{document}